\documentclass[a4paper,oneside, reqno,11pt]{amsart}
\usepackage[utf8]{inputenc}

\usepackage[top=2.3cm,bottom=2.3cm,left=2.3cm,right=2.3cm]{geometry}

 \usepackage[linkcolor = red, colorlinks = true]{hyperref}
\allowdisplaybreaks
\usepackage[dvips]{epsfig}
\usepackage{color}
\usepackage{amsgen, amstext,amsbsy,amsopn,amssymb, amsthm,mathrsfs,mathptmx,amsfonts,amssymb,amscd,amsmath,euscript,enumerate,verbatim,calc,enumitem,mathtools}
\usepackage{url}
\usepackage{hyperref}

\usepackage[normalem]{ulem}

\usepackage{soul,xcolor}

\theoremstyle{plain}
\newtheorem{theorem}{Theorem}[section]
\newtheorem{proposition}[theorem]{Proposition}

\newtheorem{lemma}[theorem]{Lemma}

\theoremstyle{definition}
\newtheorem{definition}[theorem]{Definition}
\newtheorem{example}[theorem]{Example}
\theoremstyle{remark}
\newtheorem{remark}[theorem]{\bf Remark}
\newtheorem{note}[theorem]{\bf Note}

\begin{document}

\title[]{Visualization for Petrov's odd unitary group} 

\author{A.A. Ambily}
\author{V.K. Aparna Pradeep}

\address{\newline Department of Mathematics \hfill{Department of Mathematics}
\newline Cochin University of Science and Technology \hfill{Cochin University of Science and Technology} 
\newline{Cochin University P.O.} \hfill {Cochin University P.O.}
\newline 682022, Kerala \hfill{682022, Kerala}
\newline India \hfill{India}
\newline {\it{Email: }\tt ambily@cusat.ac.in, aaambily@gmail.com }\hfill {\it{Email: }\tt aparnapradeepvk@gmail.com} }

\begin{abstract}
In this article we define a set of matrices analogous to Vaserstein-type matrices which was introduced in the paper `Serre's problem on projective modules over polynomial rings and algebraic $K$-theory' by Suslin-Vaserstein in 1976. We prove that these are elementary linear matrices.  Also, under some conditions, these matrices belong to Petrov's odd unitary group which is a generalization of all classical groups. We also prove that the group generated by these matrices is a conjugate of the Petrov's odd elementary hyperbolic unitary group when the ring is commutative.
\end{abstract}


\vspace{1mm}


\date{\today}

\maketitle
\setstcolor{red}

\section{Introduction}

In \cite{VasersteinSuslin1976}, L.N. Vaserstein and A.A. Suslin studied the freeness of projective modules over polynomial rings. They proved that a positive solution exists for Serre's problem for polynomial rings in five variables over an arbitrary field and also for polynomial rings in four variables over a principal ideal domain. For a given alternating matrix $\varphi$ of size $2n$,  Vaserstein proved the existence of two elementary matrices of size $2n-1$ which can be modified to get symplectic matrices with respect to $\varphi$. 

\vspace{2mm}

In \cite{ChattopadhyayRao2016}, P. Chattopadhyay and R.A. Rao used Vaserstein's construction for defining relative elementary symplectic group with respect to a given invertible alternating matrix. They  proved dilation principle and local-global principle for the relative elementary symplectic group and used these results for the study of symplectic transvection groups. They have also established the equality of the orbit space of a unimodular element under the action of the linear group, the symplectic group with respect to the standard symplectic form, and the symplectic group with respect to an invertible alternating matrix.

\vspace{2mm}

In \cite{Petrov2005}, V.A. Petrov introduced a new type of classical-like group known as odd unitary groups which generalizes all the classical Chevalley groups, the groups ${\rm U}_{2n+1}(R)$ of E. Abe \cite{Abe1977}, and the classical-like groups such as Bak's hyperbolic unitary groups \cite{Bak1969} and G. Tang's Hermitian groups \cite{Tang1998}. The importance of this group is that it include the odd dimensional orthogonal group ${\rm O}_{2n+1}$ as well as the groups ${\rm U}_{2n+1}(R)$ of E. Abe which was not covered in the theory of quadratic or Hermitian groups. Later A. Bak and R. Preusser \cite{BakPreusser2018} studied a subclass of Petrov’s odd unitary groups which contain all the classical Chevalley groups, and classified the E-normal subgroups of the members of this subclass. The E-normal subgroups of odd unitary groups are also studied by W. Yu, Y. Li and H. Liu in \cite{YuLiLiu2018} and also by R. Preusser in \cite{Preusser2020}.

\vspace{2mm}

In a recent preprint \cite{AmbilyRao2021},  A.A. Ambily and R.A. Rao  have defined the Vaserstein-type matrices for DSER elementary orthogonal transformations and proved that the group generated by these matrices is a congruent of the DSER group. In this paper, we do an analogous result for Petrov's odd elementary hyperbolic unitary group. Since Petrov's odd unitary group is a generalization of all classical groups and it also coincides with the DSER elementary orthogonal group in the commutative ring case with pseudo involution $a \mapsto -a$, the result in this paper generalizes analogous results in literature. 

\vspace{2mm}

Here we define two matrices, $L(v)$ and $L(v)^{*}$, similar to the construction by Vaserstein which reduces to Vaserstein-type matrices and the matrices defined by Ambily-Rao in certain special cases. We prove that for arbitrary $v \in R^{n+2m-1}$, these matrices belong to ${\rm E}_{n+2m}(R)$. We also prove that if $v=(a_1, \ldots, a_{n+2m-1}) \in R^{n+2m-1}$ satisfies the condition 
$$\overline{a}_1-a_1=\bar{1}^{-1} \overline{(a_2, \ldots, a_{n+2m-1})}\varPsi (a_2, \ldots, a_{n+2m-1})^t = q(a_2, \ldots, a_{n+2m-1}),$$ then $L(v) \in {\rm U}_{2m}(R, \mathfrak L_{\max})$ and for those $v$ satisfying $$\overline{\bar{1}a_1}-\bar{1}a_1=\overline{(a_2, \ldots, a_{n+2m-1})}(\widetilde{\psi}'_{m-1} \perp (\bar{1}^{-1})^2 \varphi)^t (a_2, \ldots, a_{n+2m-1})^t,$$ 
$L(v)^{*}$ is an isometry, but need not belongs to the odd unitary group. These results involve more calculations than the case of Bak's unitary group because of the fact that, for a transvection to be an element in Petrov's odd unitary group they have to be an isometry as well as congruent to the identity modulo the odd form parameter taken. The latter condition is an additional requirement in this group.

\vspace{2mm}

The following is the main result proved in this article.

\begin{theorem}\label{A1}
	Let $R$ be a commutative ring with unity. The group generated by $L(v)$ and $L(v)^{*}$ for $v \in R^{2m+n-1}$ is congruent to the Petrov's odd elementary hyperbolic unitary group ${\rm EU}_{2m}(R,\mathfrak L_{\max})$.
\end{theorem}

We can study the relative version of Petrov's odd unitary group using this result analogous to the study done by Chattopadhyay and Rao in  \cite{ChattopadhyayRao2016}. This theorem will help us to visualize the odd elementary hyperbolic unitary group in a simpler form and enable us to study the structure of that group in more depth.

\section{Preliminaries} 

\subsection{Odd Unitary Groups}
Odd unitary groups were defined by V.A. Petrov in \cite{Petrov2005}.
\begin{definition}
	Let $R$ be an associative ring with 1. An additive map $\sigma:R \to R$ defined by $r \mapsto \bar{r}$ satisfying the properties $\overline{r_{1}r_{2}}=\bar{r}_{2} \bar{1}^{-1} \bar{r}_{1}$ and $\bar{\bar{r}}=r$ for all $r,r_{1}, r_{2} \in R$ is called a pseudoinvolution on $R$.
\end{definition}

Let $R$ be a ring with pseudoinvolution and $V$ be a right $R$-module. A map $\langle , \rangle:V \times V \to R$ is called a sesquilinear form on $V$ if it is biadditive and satisfy the equation
$$\langle ur,vs \rangle =\bar{r}\bar{1}^{-1}\langle u,v \rangle s,$$ for all $u,v \in V$ and $r,s \in R$.

A sesquilinear form is anti-Hermitian if it also satisfies $ \langle u,v \rangle =-\overline{\langle v,u\rangle}$ for all $u,v \in V$.

\begin{definition}
	Let $V$ be a right $R$-module and $\langle, \rangle$ be an anti-Hermitian form on it. The set $\mathfrak H=V \times R$ is a group with composition defined as:
	$$(u,r)\dot{+}(v,s)=(u+v,r+s+\langle u,v \rangle).$$ Here the identity element is $(0,0)$ and the inverse of $(u,r)$ is $(-u,-r+\langle u,u\rangle)$ which is denoted by $\dot{-}(u,r)$. This group is called the {\it Heisenberg group} of the form $\langle, \rangle$.
\end{definition}

Consider the right action of $R$ on $\mathfrak H$ defined by $(u,r) \leftharpoonup s=(us, \bar{s}\bar{1}^{-1}rs).$
The {\it{trace}} map on $\mathfrak H,~ tr:\mathfrak H \to R$ is a group homomorphism defined by   
$tr((u,r))=r-\bar{r}-\langle u,u \rangle.$

Define $\mathfrak L_{\min}$ and $\mathfrak L_{\max}$ as follows.
$$\mathfrak L_{\min}=\{(0,r+\bar{r}): r \in R\} ,\hspace{2mm} \mathfrak L_{\max}=\{\zeta \in \mathfrak H: tr(\zeta)=0\}.$$

The subsets $\mathfrak L_{\min}$ and $\mathfrak L_{\max}$ are subgroups of $\mathfrak H$. These subgroups are stable under the action of $R$ and satisfies $\mathfrak L_{\min} \leq \mathfrak L_{\max}$.

\begin{definition}
	An \emph{odd form parameter} is a subgroup $\mathfrak L$ of $\mathfrak H$ satisfying $\mathfrak L_{\min} \leq \mathfrak L \leq  \mathfrak L_{\max}$ and is stable under the action of $R$.
\end{definition}

If $\mathfrak L$ is an odd form parameter corresponding to the sesquilinear form $\langle, \rangle$, then the pair $q=(\langle, \rangle, \mathfrak L)$ is called an {\it odd quadratic form} and the pair $(V,q)$ is called an {\it odd quadratic space}.

Let $(V,q)$ be an odd quadratic space. The even part of the odd form parameter $\mathfrak{L}$ is defined as $\mathfrak{L}_{ev}=\{ a \in R:(0,a) \in \mathfrak L \}$ and the even part of the quadratic space is defined as 
$$V_{ev}=\{ u \in V:(u,a) \in \mathfrak L \mbox{ for a certain $a \in R$} \}.$$ 

Let $V$ and $V'$ be two $R$-modules equipped with form parameters $\mathfrak L$ and $\mathfrak L'$ respectively. An isometry $f:V \to V'$ preserves the form parameters if $(f(v),r) \in \mathfrak L'$ for all $(v,r) \in \mathfrak L$. Two such isometries $f$ and $g$ are said to be equivalent modulo $\mathfrak L'$ if $(f(v)-g(v),\langle g(v)-f(v), g(v)\rangle) \in \mathfrak L'$ for all $v \in V$. It can be verified that this relation is an equivalence relation on the set of all isometries from $V$ to $V'$.
If $f$ and $g$ are equivalent modulo $\mathfrak L$, we denote it by $f \cong g({\rm mod} ~ \mathfrak L)$.
\begin{definition}
	The \emph{odd unitary group} denoted by ${\rm U}(V,q)$ is the group of all bijective isometries $f$ on $V$ such that $f \cong e ({\rm mod} ~ \mathfrak L),$ where $e$ is the identity map on $V$.
\end{definition}

Let $u,v$ be elements of an odd quadratic space $V$ and $r$ be an element of $R$ such that $\langle u,v\rangle=0,~ (u,0) \in \mathfrak L$ and $(v,r) \in \mathfrak L$. The transformations on $V$ of the form   $T_{u,v}(a)$ defined by
$$T_{u,v}(r)(w)=w+u\bar{1}^{-1}(\langle v,w \rangle +r\langle u,w \rangle )+v\langle u,w \rangle,$$ are known as {\it Eichler-Siegel-Dickson transvections}.

\begin{lemma}[\cite{Petrov2005}, Lemma 1]
	The transvections $T_{u,v}(a)$ lie in ${\rm U}(V,q)$.
\end{lemma}

A pair of  vectors $(u,v)$ such that $\langle u,v \rangle =1,~(u,0) \in \mathfrak L,~ (v,0) \in \mathfrak L$ is called a hyperbolic pair. 
Let $\mathbb H$ be odd quadratic space spanned by vectors $e_{1}$ and $e_{-1}$ such that $\langle e_{1},e_{-1} \rangle =1$, equipped with the odd form parameter $\mathfrak L$ generated by $(e_{1},0)$ and $(e_{-1},0)$. Denote the orthogonal sum of $m$ copies of $\mathbb H$ by $\mathbb H^{m}$. 

\vspace{2mm}
Consider an odd quadratic space $V_{0}$ equipped with an odd quadratic form $q_0=(\langle, \rangle_0, \mathfrak L_0)$. The orthogonal sum $V=\mathbb H^{m} \oplus V_0$ is called the {\it odd hyperbolic unitary group} and is denoted by ${\rm U}_{2m}(R, \mathfrak L)$.

\begin{definition}
	The greatest number $n$ satisfying the condition that there exist $n$ mutually orthogonal hyperbolic pairs in $(V,q)$ is called the Witt index of $(V,q)$, denoted by ${\rm ind(V,q)}$.
\end{definition}

Let $V$ be an odd quadratic space that has Witt index at least $n$. Then we can choose an embedding of $\mathbb H^{n}$ to $V$. Fix any such embedding. Then we have elements $\{e_{i}\}_{i=1, \ldots, n,-n, \ldots, -1}$ in $V$ such that $\langle e_{i},e_{j}\rangle =0$ for $i \neq -j,$ $\langle e_{i},e_{-i} \rangle =1$ for $i=1, \ldots, n$ and $(e_{i},0) \in \mathfrak L$. Let $V_{0}$ denote the orthogonal complement of $\langle e_{1}, \ldots, e_{n},e_{-n}, \ldots, e_{-1} \rangle$ in $V$ and $\langle, \rangle_{0}$ and $\mathfrak L_{0}$ be the restrictions of $\langle, \rangle$ and  $\mathfrak L$ respectively to $V_{0}$. Then $V$ is isometric to the odd hyperbolic space $\mathbb H^{n} \oplus V_{0}$ with the quadratic form $(\langle, \rangle_{0}, \mathfrak L_{0})$. Thus ${\rm U}(V,q)={\rm U}_{2n}(R, \mathfrak L_{0})$.

\vspace{2mm}

Consider a non-reduced root system of type $BC_{l}$ consisting of vectors $\chi_{i} \pm \chi_{j},~ \pm \chi_{i}$ and $\pm 2 \chi_{i}$, where $\{\chi_{i}\}_{i=1, \ldots ,n}$ is an orthonormal basis for the $n$-dimensional Euclidean space. 
Corresponding to the roots we can define elementary transvections as follows.
\begin{align*}
&T_{\chi_{i}-\chi_{j}}(r)={T_{i,j}(r)=T_{e_{-j},-e_{i}r\varepsilon_{j}}(0),~ r \in R},\\
&T_{\chi_{i}}(u,r)={T_{i}(u,r)=T_{e_{i},u\varepsilon_{-i}}(-\bar{\varepsilon}_{-i}\bar{1}^{-1}r\varepsilon_{-i}),~ (u,r) \in \mathfrak L},\\
&T_{2\chi_{i}}(r)=T_{i}(0,r),~ r \in \mathfrak L_{ev},
\end{align*}
where $i,j \in \{-m, \ldots 1, -1, \ldots, m\}$ and $\varepsilon_{i}= \begin{cases}
\bar{1}^{-1},~ i>0\\
-1, ~ i<0.
\end{cases}$

\vspace{2mm}

The odd elementary hyperbolic unitary group ${\rm EU}_{2m}(R, \mathfrak L)$ is the subgroup of ${\rm U}_{2m}(R, \mathfrak L)$ generated by all elementary transvections. 

\section{Matrix form of elementary transvections}

Let $R$ be a commutative ring with identity. Here we write the matrices for the generators of ${\rm EU}_{2m}(R,\mathfrak{L})$ when the odd quadratic space $V_0$ is free of rank $n$. Let $\{ v_{1},\ldots, v_{n},e_{1},e_{-1},\ldots,e_{m},e_{-m} \}$ be a basis for the odd hyperbolic unitary space $V = \mathbb H^{m} \oplus V_{0}$.   Let the unitary space $V$ be equipped with  anti-Hermitian form $\langle , \rangle$, i.e., $\langle u,v \rangle=-\overline{\langle v,u \rangle}$ for all $u,v \in V$. 
We use the following important result proved by Petrov in \cite{Petrov2005}.\begin{proposition}[\cite{Petrov2005}, Proposition 1] \label{B}
	The group ${\rm EU}_{2l}(R, \mathfrak L)$ coincides with the group generated by all elements of the form $T_{e_{\pm 1},v}(a),$ where $\langle e_1, v \rangle=\langle e_{-1}, v \rangle=0$ and $(v,a) \in \mathfrak L$.
\end{proposition}

Let $u=t_{1}v_{1}+ \ldots +t_{n}v_{n}+b_{1}e_{1}+b_{-1}e_{-1}+\ldots+b_{m}e_{m}+b_{-m}e_{-m}$ be an arbitrary element of $V$. For $i \in \{1, \ldots, m\}$, we have $\varepsilon_{-i}=-1$ and therefore $T_{i}(u,a)=T_{e_{i},-u}(-a)$. Also for $i \in \{-m, \ldots, -1\}$, we have $\varepsilon_{-i}=\bar{1}^{-1}$ and therefore $T_{i}(u,a)=T_{e_{i},u \bar{1}^{-1}}(-a)$. Thus for arbitrary $(v,a) \in \mathfrak L$ with $\langle e_1, v \rangle=\langle e_{-1}, v \rangle=0$ we have $T_{e_{1},v}(a)=T_1(-v,-a)$ and $T_{e_{-1},v}(a)=T_{-1}(v \bar{1},-a)$. Thus we can restate the above proposition as follows.

\begin{proposition}\label{B1}
	The group ${\rm EU}_{2l}(R, \mathfrak L)$ coincides with the group generated by all elements of the form $T_{\pm 1}(v,a)$, where $\langle e_1, v \rangle=\langle e_{-1}, v \rangle=0$ and $(v,a) \in \mathfrak L$.
\end{proposition}

\begin{note}
The pseudoinvolution we have taken satisfies the properties $\overline{r_{1}r_{2}}=\bar{r}_{2} \bar{1}^{-1} \bar{r}_{1}$ and $\bar{\bar{r}}=r$ for all $r,r_{1}, r_{2} \in R$. Therefore we have $\bar{1} \cdot \bar{1}=\overline{\bar{1}^{-1}}$.
\end{note}

Now we write the matrices for the generators $T_{\pm 1}(v,a)$, of the Petrov's odd elementary hyperbolic unitary group.

\vspace{3mm}
For computing $T_{1}(u,a)=T_{e_{1},-u}(-a)$, the element $u$ should satisfy $\langle e_{1},-u \rangle=0$ which implies that $b_{-1}=0$. Therefore $u$ has the form 
$$u=t_{1}v_{1}+ \ldots +t_{n}v_{n}+b_{1}e_{1}+b_{2}e_{2}+b_{-2}e_{-2}+\ldots+b_{m}e_{m}+b_{-m}e_{-m}.$$

We have $T_{1}(u,a)(w)=T_{e_{1},-u}(-a)(w)=w+e_{1} \bar{1}^{-1}(\langle -u,w \rangle-a\langle e_1,w \rangle)-u \langle e_1,w \rangle$. Thus we get
\begin{align*}
T_{1}(u,a)(v_k)&=v_k+e_1 \bar{1}^{-1} \langle -v_1 t_1 -\dots -v_n t_n, v_k\rangle \\
&=v_k-e_1 (\bar{1}^{-1})^2 ( \overline{t}_1 \langle v_1,v_k \rangle + \ldots +\overline{t}_n \langle v_n, v_k\rangle), \mbox{ for }  k \in \{1, \ldots n \}.
\end{align*}
Also $T_1(u,a)(e_1)=e_1$ and $$T_1(u,a)(e_{-1})=-t_{1}v_{1}- \ldots -t_{n}v_{n}-e_{1}(\overline{b_1}(\bar{1}^{-1})^2+a \bar{1}^{-1}+b_1)+e_{-1}-b_2 e_2- b_{-2} e_{-2}-\ldots-b_{m}e_{m}-b_{-m}e_{-m}.$$ 
For $k \in \{2, \ldots , m\}$, we get $T_1(u,a)(e_k)=e_k+e_1 \bar{1}^{-1}\overline{b_{-k}}$ and for $k \in \{-m, \ldots , -2\}$,  $T_1(u,a)(e_k)=e_k-e_1 (\bar{1}^{-1})^2 \overline{b_{-k}}$.

\vspace{3mm}

Therefore the matrix of $T_1(u,a)$ is of the form
$${\scriptsize
	\begin{pmatrix}
	1 & \ldots & 0 & 0 & -t_1 & 0 & 0 & \ldots & 0 & 0\\
	\vdots & \vdots & \vdots & \vdots & \vdots & \vdots & \vdots & \vdots & \vdots & \vdots\\
	0 & \ldots & 1 & 0 & -t_n & 0 & 0 & \ldots & 0 & 0\\
	-(\bar{1}^{-1})^2 (\overline{t}_1 \langle v_1,v_1 \rangle & \ldots & -(\bar{1}^{-1})^2 ( \overline{t}_1 \langle v_1,v_n \rangle & 1 & -(\overline{b}_1(\bar{1}^{-1})^2 & \bar{1}^{-1}\overline{b}_{-2} & -(\bar{1}^{-1})^2 \overline{b}_{2} & \ldots & \bar{1}^{-1}\overline{b}_{-m} & -(\bar{1}^{-1})^2 \overline{b}_{m}\\
	+ \ldots +\overline{t}_n \langle v_n, v_1\rangle) & & + \ldots +\overline{t}_n \langle v_n, v_n\rangle) & & \hspace{8mm}+a \bar{1}^{-1}+b_1) & & & & &\\
	0 & \ldots & 0 & 0 & 1 & 0 & 0 & \ldots & 0 & 0\\
	0 & \ldots & 0 & 0 & -b_2 & 1 & 0 & \ldots & 0 & 0\\
	0 & \ldots & 0 & 0 & -b_{-2} & 0 & 1 & \ldots & 0 & 0\\
	\vdots & \vdots & \vdots & \vdots & \vdots & \vdots & \vdots & \vdots & \vdots & \vdots\\
	0 & \ldots & 0 & 0 & -b_m & 0 & 0 & \ldots & 0 & 0\\
	0 & \ldots & 0 & 0 & -b_{-m} & 0 & 0 & \ldots & 0 & 1
	\end{pmatrix}}.$$

\vspace{2mm}

Now for computing $T_{-1}(u,a)=T_{e_{-1},u \bar{1}^{-1}}(-a)$, the element $u$ should satisfy $\langle e_{-1},u \bar{1}^{-1} \rangle=0$ which implies $b_{1}=0$. Therefore $u$ will be of the form 
\vspace{-2mm}
$$u=t_{1}v_{1}+ \ldots +t_{n}v_{n}+b_{-1}e_{-1}+b_{2}e_{2}+b_{-2}e_{-2}+\ldots+b_{m}e_{m}+b_{-m}e_{-m}.$$ 
We have $$T_{-1}(u,a)(w)=T_{e_{-1},u \bar{1}^{-1}}(-a)(w)=w+e_{-1} \bar{1}^{-1}(\langle u \bar{1}^{-1},w \rangle-a\langle e_{-1},w \rangle)+u \bar{1}^{-1} \langle e_{-1},w \rangle.$$ 
\begin{align*}
\mbox{Thus }~~T_{-1}(u,a)(v_k)&=v_k+e_{-1} \bar{1}^{-1} \langle (v_1 t_1 +\dots +v_n t_n)\bar{1}^{-1}, v_k\rangle \\
&=v_k+e_{-1} \bar{1}^{-1} ( \overline{t}_1 \langle v_1,v_k \rangle + \ldots +\overline{t}_n \langle v_n, v_k\rangle), \mbox{ for }  k \in \{1, \ldots n \},
\end{align*} 
\begin{align*}
 T_{-1}(u,a)(e_{1})=-t_{1}v_{1}&- \ldots -t_{n}v_{n}+e_{1}+e_{-1}(-\overline{b_{-1}}+a \bar{1}-b_{-1})-b_2 e_2- b_{-2} e_{-2}\\
 &-\ldots-b_{m}e_{m}-b_{-m}e_{-m}
\end{align*}
 and $T_{-1}(u,a)(e_{-1})=e_{-1}$.
 
 \vspace{3mm}
\noindent Also  $T_{-1}(u,a)(e_k)=\begin{cases}
                         e_k-e_{-1}\overline{b}_{-k}, \mbox{ for } k \in \{2, \ldots , m\},\\
                         e_k+e_{-1} \bar{1}^{-1}\overline{b}_{-k} \mbox{ for } k \in \{-m, \ldots , -2\}.
                        \end{cases}$
\vspace{3mm}

Therefore the matrix of $T_{-1}(u,a)$ is of the form
$${\scriptsize
	\begin{pmatrix}
	1 & \ldots & 0 & -t_1 & 0 & 0 & 0 & \ldots & 0 & 0\\
	\vdots & \vdots & \vdots & \vdots & \vdots & \vdots & \vdots & \vdots & \vdots & \vdots\\
	0 & \ldots & 1 & -t_n & 0 & 0 & 0 & \ldots & 0 & 0\\
	0 & \ldots & 0 & 1 & 0 & 0 & 0 & \ldots & 0 & 0\\
	\bar{1}^{-1} ( \overline{t}_1 \langle v_1,v_1 \rangle & \ldots & \bar{1}^{-1} ( \overline{t}_1 \langle v_1,v_n \rangle & -\overline{b}_{-1}+a-b_{-1} & 1 & -\overline{b}_{-2} & \bar{1}^{-1} \overline{b}_{2} & \ldots & -\overline{b}_{-m} & \bar{1}^{-1} \overline{b}_{m}\\
	+ \ldots +\overline{t}_n \langle v_n, v_1\rangle) & & + \ldots +\overline{t}_n \langle v_n, v_n\rangle) & & & & & \\
	0 & \ldots & 0 & -b_2 & 0 & 1 & 0 & \ldots & 0 & 0\\
	0 & \ldots & 0 & -b_{-2} & 0 & 0 & 1 & \ldots & 0 & 0\\
	\vdots & \vdots & \vdots & \vdots & \vdots & \vdots & \vdots & \vdots & \vdots & \vdots\\
	0 & \ldots & 0 & -b_m & 0 & 0 & 0 & \ldots & 0 & 0\\
	0 & \ldots & 0 & -b_{-m} & 0 & 0 & 0 & \ldots & 0 & 1
	\end{pmatrix}}.$$

	\vspace{3mm}
Now we consider the case when $R$ is a commutative ring and the involution is $a \mapsto \bar{a}=-a$. In this case, the above matrices reduces to the following form.

$${\scriptsize T_{1}(u,a)=\begin{pmatrix}
	1 & \ldots & 0 & 0 & -t_1 & 0 & 0 & \ldots & 0 & 0\\
	\vdots & \vdots & \vdots & \vdots & \vdots & \vdots & \vdots & \vdots & \vdots & \vdots\\
	0 & \ldots & 1 & 0 & -t_n & 0 & 0 & \ldots & 0 & 0\\
	t_1 \langle v_1,v_1 \rangle & \ldots & t_1 \langle v_1,v_n \rangle & 1 & a & b_{-2} & b_{2} & \ldots & b_{-m} & b_{m}\\
	+ \ldots +t_n \langle v_n, v_1\rangle & & + \ldots +t_n \langle v_n, v_n\rangle & & & & & & &\\
	0 & \ldots & 0 & 0 & 1 & 0 & 0 & \ldots & 0 & 0\\
	0 & \ldots & 0 & 0 & -b_2 & 1 & 0 & \ldots & 0 & 0\\
	0 & \ldots & 0 & 0 & -b_{-2} & 0 & 1 & \ldots & 0 & 0\\
	\vdots & \vdots & \vdots & \vdots & \vdots & \vdots & \vdots & \vdots & \vdots & \vdots\\
	0 & \ldots & 0 & 0 & -b_m & 0 & 0 & \ldots & 0 & 0\\
	0 & \ldots & 0 & 0 & -b_{-m} & 0 & 0 & \ldots & 0 & 1
	\end{pmatrix}}, \mbox{ and }$$

\vspace{3mm}

$$ {\scriptsize T_{-1}(u,a)=\begin{pmatrix}
	1 & \ldots & 0 & -t_1 & 0 & 0 & 0 & \ldots & 0 & 0\\
	\vdots & \vdots & \vdots & \vdots & \vdots & \vdots & \vdots & \vdots & \vdots & \vdots\\
	0 & \ldots & 1 & -t_n & 0 & 0 & 0 & \ldots & 0 & 0\\
	0 & \ldots & 0 & 1 & 0 & 0 & 0 & \ldots & 0 & 0\\
	t_1 \langle v_1,v_1 \rangle & \ldots & t_1 \langle v_1,v_n \rangle & a & 1 & b_{-2} & b_{2} & \ldots & b_{-m} & b_{m}\\
	+ \ldots + t_n \langle v_n, v_1\rangle& & + \ldots +t_n \langle v_n, v_n\rangle & & & & & & &\\
	0 & \ldots & 0 & -b_2 & 0 & 1 & 0 & \ldots & 0 & 0\\
	0 & \ldots & 0 & -b_{-2} & 0 & 0 & 1 & \ldots & 0 & 0\\
	\vdots & \vdots & \vdots & \vdots & \vdots & \vdots & \vdots & \vdots & \vdots & \vdots\\
	0 & \ldots & 0 & -b_m & 0 & 0 & 0 & \ldots & 0 & 0\\
	0 & \ldots & 0 & -b_{-m} & 0 & 0 & 0 & \ldots & 0 & 1
	\end{pmatrix}}.$$

\vspace{3mm}

\section{Visualization of odd elementary hyperbolic unitary group}

Let $\varphi$ be an invertible alternating matrix of size $2n$ of the form     
$\begin{pmatrix}
0 & -c\\
c^{t} & \mu
\end{pmatrix}$ and $\varphi^{-1}$ is of the form $\begin{pmatrix}
0 & d\\
-d^{t} & \rho
\end{pmatrix}$, where $c,d \in R^{2n-1}$. For $v \in R^{2n-1}$, Vaserstein (in \cite{VasersteinSuslin1976}, Lemma 5.4) considered the matrices 
$$\alpha =I_{2n-1}+d^t v \mu \mbox{ and } \beta = I_{2n-1} - \rho v^t c$$
and proved that the matrices 
$$ \begin{pmatrix}
1 & 0\\
v^{t} & \alpha
\end{pmatrix}
\mbox{ and } \begin{pmatrix}
1 & v\\
0 & \beta
\end{pmatrix}$$ 
belongs to ${\rm E}_{2n}(R) \cap {\rm Sp}_{\varphi}(R)$, where ${\rm Sp}_{\varphi}(R)$ is the group of all symplectic matrices with respect to the form $\varphi$. 

\vspace{2mm}
Here we define analogous matrices for Petrov's odd elementary hyperbolic unitary group. The form matrix for Petrov's group is $\varPsi=\widetilde{\psi}_{m} \perp \varphi$, where $\varphi$ is the form matrix for $V_0$ and $\widetilde{\psi}_{r}$ is defined as follows.

$$\widetilde{\psi_{1}}=\begin{pmatrix}
0 & 1\\
-\bar{1} & 0
\end{pmatrix}, 
\mbox{ and } \widetilde{\psi}_{r}=\widetilde{\psi}_{1} \perp \widetilde{\psi}_{r-1} \mbox{ for } r >1.$$ 
Then $\varPsi$ can be written in the form 
$\begin{pmatrix}
0 & c\\
-\bar{1} c^{t} & \mu
\end{pmatrix}$, where 
$c=\begin{pmatrix}
1 & \ldots & 0
\end{pmatrix} \in R^{n+2m-1}$
and $\mu$ is the $(n+2m) \times (n+2m)$ matrix given by $\mu=\begin{pmatrix}
0 & 0 & 0\\
0 & \widetilde{\psi}_{m-2} & 0\\
0 & 0 & \varphi
\end{pmatrix}$. 
Also we have $\varPsi^{-1}$ is of the form
$\begin{pmatrix}
0 & d\\
-\bar{1} d^{t} & \rho
\end{pmatrix}$, where 
$d=\begin{pmatrix}
-\bar{1}^{-1} & \ldots & 0
\end{pmatrix} \in R^{n+2m-1}$ and 
$\rho=\begin{pmatrix}
0 & 0 & 0\\
0 & \widetilde{\psi}'_{m-2} & 0\\
0 & 0 & \varphi^{-1}
\end{pmatrix}$, where $\widetilde{\psi}'_{m}$ is defined as follows.

$$\widetilde{\psi}'_{1}=\begin{pmatrix}
0 & -\bar{1}^{-1}\\
1 & 0
\end{pmatrix}
\mbox{ and } \widetilde{\psi}'_{r}=\widetilde{\psi}'_{1} \perp \widetilde{\psi}'_{r-1}, \mbox{  for  } r >1.$$   

For $v=(a_1, \ldots, a_{n+2m-1}) \in R^{n+2m-1}$, define the Vaserstein type matrices $L(v)$ and $L(v)^{*}$ as 
\vspace{-2mm}      
$$L(v)=\begin{pmatrix}
1 & 0\\
v^{t} & \alpha
\end{pmatrix}
\mbox{ and } L(v)^{*}=\begin{pmatrix}
1 & v\\
0 & \beta
\end{pmatrix},$$       
\vspace{-1mm}       
where $\alpha =I_{n+2m-1}+d^t \overline{v} \mu$ and $\beta = I_{n+2m-1} - \bar{1}^{-1} \rho \overline{v}^t c$.

\vspace{3mm}

By direct computations, we get the matrices $\alpha$ and $\beta$ corresponding to the form $\Psi$ as follows:
$$\alpha = {\scriptsize{\begin{pmatrix}
		1 & \overline{a}_3 & -\bar{1}^{-1} \overline{a}_2 & \ldots & \overline{a}_{2m-1} & -\bar{1}^{-1} \overline{a}_{2m-2} & -\bar{1}^{-1}(\overline{a}_{2m}\varphi_{11}+\overline{a}_{2m+1} \varphi_{21} & \ldots & -\bar{1}^{-1}(\overline{a}_{2m}\varphi_{1n}+\overline{a}_{2m+1} \varphi_{2n}\\
		& & & & & & +\ldots+\overline{a}_{2m-1+n} \varphi_{n1}) & & +\ldots+\overline{a}_{2m-1+n} \varphi_{nn})\\ 
		0 & 1 & 0 & \ldots & 0 & 0 & 0 & \ldots & 0\\
		0 & 0 & 1 & \ldots & 0 & 0 & 0 & \ldots & 0\\
		\vdots & \vdots & \vdots & \vdots & \vdots & \vdots & \vdots & \vdots & \vdots\\
		0 & 0 & 0 & \ldots & 0 & 0 & 0 & \ldots & 1
		\end{pmatrix}}},$$

$$\mbox{ and } \beta = \scriptsize{\begin{pmatrix}
	1 & 0 & 0 & \ldots & 0\\
	(\bar{1}^{-1})^2 \overline{a}_3 & 1 & 0 & \ldots & 0\\
	-\bar{1}^{-1} \overline{a}_2 & 0 & 1 & \ldots & 0\\
	\vdots & \vdots & \vdots & \vdots & \vdots\\
	(\bar{1}^{-1})^2 \overline{a}_{2m-1} & 0 & 0 & \ldots & 0\\
	-\bar{1}^{-1} \overline{a}_{2m-2} & 0 & 0 & \ldots & 0\\
	-\bar{1}^{-1} (\overline{a}_{2m}(\varphi^{-1})_{11}+\overline{a}_{2m+1} (\varphi^{-1})_{12}+\ldots+\overline{a}_{2m-1+n} (\varphi^{-1})_{1n}) & 0 & 0 & \ldots & 0\\
	\vdots & \vdots & \vdots & \vdots & \vdots \\
	-\bar{1}^{-1} (\overline{a}_{2m}(\varphi^{-1})_{n1}+\overline{a}_{2m+1} (\varphi^{-1})_{n2}+\ldots+\overline{a}_{2m-1+n} (\varphi^{-1})_{nn} & 0 & 0 & \ldots & 1
	\end{pmatrix}}.$$

Therefore, we have 
$$\hspace{-8mm}L(v)= {\scriptsize{\begin{pmatrix}
		1 & 0 & 0 & 0 & \ldots & 0 & 0 & 0 & \ldots & 0\\
		a_1 & 1 & \overline{a}_3 & -\bar{1}^{-1} \overline{a}_2 & \ldots &  \overline{a}_{2m-1} & -\bar{1}^{-1} \overline{a}_{2m-2} & -\bar{1}^{-1}(\overline{a}_{2m}\varphi_{11}+\overline{a}_{2m+1} \varphi_{21} & \ldots & -\bar{1}^{-1}(\overline{a}_{2m}\varphi_{1n}+\overline{a}_{2m+1} \varphi_{2n}\\
		& & & & & & & +\ldots+\overline{a}_{2m-1+n} \varphi_{n1}) & & +\ldots+\overline{a}_{2m-1+n} \varphi_{nn})\\ 
		a_2 & 0 & 1 & 0 & \ldots & 0 & 0 & 0 & \ldots & 0\\
		a_3 & 0 & 0 & 1 & \ldots & 0 & 0 & 0 & \ldots & 0\\
		\vdots & \vdots & \vdots & \vdots & \vdots & \vdots & \vdots & \vdots & \vdots & \vdots\\
		a_{n+2m-1} & 0 & 0 & 0 & \ldots & 0 & 0 & 0 & \ldots & 1
		\end{pmatrix}}},$$

$$\mbox{ and }L(v)^{*}={\scriptsize{\begin{pmatrix}
		1 & a_1 & a_2 & 0 & \ldots & a_{n+2m-1}\\
		0 & 1 & 0 & 0 & \ldots & 0\\
		0 & (\bar{1}^{-1})^2 \overline{a}_3 & 1 & 0 & \ldots & 0\\
		0 & -\bar{1}^{-1} \overline{a}_2 & 0 & 1 & \ldots & 0\\
		\vdots & \vdots & \vdots & \vdots & \vdots & \vdots\\
		0 & (\bar{1}^{-1})^2 \overline{a}_{2m-1} & 0 & 0 & \ldots & 0\\
		0 & -\bar{1}^{-1} \overline{a}_{2m-2} & 0 & 0 & \ldots & 0\\
		0 & -\bar{1}^{-1} (\overline{a}_{2m}(\varphi^{-1})_{11}+\overline{a}_{2m+1} (\varphi^{-1})_{12}+\ldots+\overline{a}_{2m-1+n} (\varphi^{-1})_{1n}) & 0 & 0 & \ldots & 0\\
		\vdots & \vdots & \vdots & \vdots & \vdots & \vdots \\
		0 & -\bar{1}^{-1} (\overline{a}_{2m}(\varphi^{-1})_{n1}+\overline{a}_{2m+1} (\varphi^{-1})_{n2}+\ldots+\overline{a}_{2m-1+n} (\varphi^{-1})_{nn} & 0 & 0 & \ldots & 1
		\end{pmatrix}}}.$$

\vspace{3mm}

By an argument analogous to Lemma 5.4 in \cite{VasersteinSuslin1976}, we can prove that both $L(v)$ and $L(v)^{*}$ are elementary matrices. 

\begin{theorem}
	For $v=(a_1, \ldots, a_{n+2m-1}) \in R^{n+2m-1}$, we have
	$L(v),~L(v)^{*} \in {\rm E}_{n+2m}(R).$
\end{theorem}
\begin{proof}
	By direct computation, we can verify that $\bar{v}\mu d^t =0$. Therefore $1+\bar{v}\mu d^t =1$. By \cite[Lemma~ 2.2]{VasersteinSuslin1976}, we have $\alpha=1+d^t \bar{v}\mu \in {\rm GL}_{n+2m}(R)$ and
	$\begin{pmatrix}
	1 & 0\\
	0 & \alpha
	\end{pmatrix} \in {\rm E}_{n+2m}(R)$. Now by using \cite[Lemma ~2.1]{VasersteinSuslin1976}, we get 
	$\begin{pmatrix}
	1 & 0\\
	v^t & I_{n+2m-1}
	\end{pmatrix} \in {\rm E}_{n+2m}(R)$. Therefore we have
	$$\begin{pmatrix}
	1 & 0\\
	v^t & \alpha
	\end{pmatrix}
	=\begin{pmatrix}
	1 & 0\\
	v^t & I_{n+2m-1}
	\end{pmatrix}
	\begin{pmatrix}
	1 & 0\\
	0 & \alpha
	\end{pmatrix}\in {\rm E}_{n+2m}(R).$$
\end{proof}
The matrices $L(v)$ and $L(v)^{*}$ are defined for arbitrary $v=(a_1, \ldots, a_{n+2m-1}) \in R^{n+2m-1}$. Under some condition on the components of $v$, we get the following result.
\begin{theorem}\label{C}
	We have $L(v) \in {\rm U}_{2m}(R, \mathfrak L_{\max})$, if $v=(a_1, \ldots, a_{n+2m-1}) \in R^{n+2m-1}$ satisfies the condition 
	\begin{equation}\label{D}
	\overline{a}_1-a_1=\bar{1}^{-1} \overline{(a_2, \ldots, a_{n+2m-1})}\varPsi (a_2, \ldots, a_{n+2m-1})^t = q(a_2, \ldots, a_{n+2m-1}).
	\end{equation}
	Also, $L(v)^{*}$ preserves the form matrix $\varPsi$, for $v$ satisfying the condition
	\begin{equation}\label{E}
	\overline{\bar{1}a_1}-\bar{1}a_1=\overline{(a_2, \ldots, a_{n+2m-1})}(\widetilde{\psi}'_{m-1} \perp (\bar{1}^{-1})^2 \varphi)^t (a_2, \ldots, a_{n+2m-1})^t.
	\end{equation}
\end{theorem}
\begin{proof}
	By direct calculation we can verify that under the condition given in Equation (\ref{D}), $$\bar{1}^{-1} \overline{L(v)}^{t} \varPsi L(v) = \varPsi.$$ 
	That is, $L(v)$ preserves the form matrix for Petrov's group. Now to prove that $L(v)$ belongs to ${\rm U}_{2m}(R, \mathfrak L_{\max})$, it suffices to prove that $L(v) \cong I_{n+2m}$ (mod $\mathfrak{L}_{\max}$).
	That is $(L(v)(w)-w, \langle w-L(v)(w),w \rangle) \in \mathfrak{L}_{\max}$, for all $w \in V$.
	
	\vspace{2mm}
	For $w=(w_1, \ldots, w_{n+2m}) \in V$, we have 
	$$L(v)(w)=\begin{pmatrix}
	w_1\\
	x\\
	a_2 w_1+w_3\\
	a_3 w_1+w_4\\
	\vdots \\
	a_{n+2m-1}w_1+w_{n+2m}
	\end{pmatrix},
	$$ 
	\begin{align*}
	\mbox{where }x&=a_1 w_1+w_2+\overline{a}_{3}w_3-\bar{1}^{-1}\overline{a}_2 w_4+ \ldots +\overline{a}_{2m-1} w_{2m-1}\\
	&\hspace{1cm}-\bar{1}^{-1}\overline{a}_{2m-2}w_{2m} -\bar{1}^{-1}(\overline{a}_{2m}\varphi_{11}+\overline{a}_{2m+1} \varphi_{21}+\ldots+\overline{a}_{2m-1+n} \varphi_{n1}) w_{2m+1}\\
	&\hspace{2cm}- \ldots -\bar{1}^{-1}(\overline{a}_{2m}\varphi_{1n}+\overline{a}_{2m+1} \varphi_{2n}+\ldots+\overline{a}_{2m-1+n} \varphi_{nn})w_{n+2m}.
	\end{align*}
	Let $u=L(v)(w)-w$. Then $$u= \begin{pmatrix}
	0\\
	x_1\\
	a_2 w_1\\
	a_3 w_1\\
	\vdots \\
	a_{n+2m-1}w_1
	\end{pmatrix},
	$$ where $x_1=x-w_2.$ Now, let $a=\langle w-L(v)(w),w \rangle.$ By computations we can see that 
	\begin{align*}
	a&=\bar{1}^{-1}\overline{w}_1 \overline{a}_1 w_1 + (s_1+\overline{s}_1)- (s_2+\overline{s}_2)+ \ldots + (s_{2m-3}+ \overline{s}_{2m-3})- (s_{2m-2}+\overline{s}_{2m-2})\\
	&\hspace{2cm}- (s_{2m-1}+\overline{s}_{2m-1})- (s_{2m}+\overline{s}_{2m})- \ldots -(s_{2m+n}+\overline{s}_{2m+n}),
	\end{align*}
	where $s_k=
	\begin{cases}
	\bar{1}^{-1}\overline{w}_{k+2} a_{2k-1} w_1 , \mbox{ for } 1 \leq k \leq 2m-2 \mbox{ and $k$ is odd},\\
	\overline{w}_{k+2} a_{k} w_1, \mbox{ for } 1 \leq k \leq 2m-2 \mbox{ and $k$ is even},\\
	\bar{1}^{-1}\overline{w}_{k+2} (\overline{\varphi}_{1,k-(2m-2)} a_{2m}+\overline{\varphi}_{2,k-(2m-2)} a_{2m+1}+ \ldots +\overline{\varphi}_{n,k-(2m-2)} a_{n+2m-1})w_1, \\
	\hspace{6cm}\mbox{ for } 2m-1 \leq k \leq 2m+n-2.
	\end{cases}$
	
	\vspace{3mm}
	Thus $a-\overline{a}=\bar{1}^{-1}\overline{w}_1 (\overline{a}_1-a_1) w_1.$ Also we can verify that 
	$$\langle u,u \rangle=\bar{1}^{-1}\overline{w}_1 q(a_2,\ldots,a_{n+2m-1}) w_1=\bar{1}^{-1}\overline{w}_1 (\overline{a}_1-a_1) w_1.$$ Therefore $a-\overline{a}=\langle u,u \rangle$, which implies $(u,a) \in \mathfrak L_{\max}$. Thus $L(v) \in {\rm U}_{2m}(R, \mathfrak L_{\max}).$
	
	\vspace{2mm}
	Now assume that $v$ satisfies the condition given in Equation (\ref{E}). Then we can see that
	$$\bar{1}^{-1} \overline{L(v)^{*}}^{t} \varPsi L(v)^{*}=\varPsi,$$ which shows that $L(v)^{*}$ preserves the form matrix $\varPsi.$
\end{proof}

\begin{remark}\label{F}
	By assuming the condition given in Equation (\ref{E}) and doing similar calculations as in the proof of Theorem \ref{C}, for $u_1=L(v)^{*}(w)-w$ and $b=\langle w-L(v)^{*}(w),w \rangle$, we get $$\langle u_1,u_1 \rangle=\overline{w}_2 (a_1 - (\bar{1}^{-1})^{2}\overline{a}_1) w_2  ~~\mbox{ and } ~~
	b-\bar{b}= \overline{w}_2 (a_1 - (\bar{1}^{-1})^{2}\overline{a}_1) w_2+(1-(\bar{1}^{-1})^2)(s+\overline{s}),$$
	where $s=\overline{w}_2 \begin{pmatrix}
	a_{2m} & a_{2m+1} & \ldots & a_{n+2m-1}
	\end{pmatrix} 
	\begin{pmatrix}
	w_{2m+1}\\
	w_{2m+2}\\
	\vdots \\
	w_{n+2m}
	\end{pmatrix}$.
	
	\vspace{3mm}
	Thus $(u_1,b) \notin \mathfrak L_{\max}$ and hence $L(v)^{*} \notin {\rm U}_{2m}(R, \mathfrak L_{\max})$ in general. In the special case in which the involution satisfies $(\bar{1}^{-1})^2=1$, we have both $L(v)$ and $L(v)^{*}$ belongs to ${\rm U}_{2m}(R, \mathfrak L_{\max})$.
	
\end{remark}

Next, we shall consider the group generated by these two type of matrices and we have the following theorem.
\begin{theorem}\label{A}
	Let $R$ be a commutative ring with unity. The group generated by $L(v)$ and $L(v)^{*}$ for $v \in R^{2m+n-1}$ is congruent to the Petrov's odd elementary hyperbolic unitary group ${\rm EU}_{2m}(R,\mathfrak L_{\max})$.
\end{theorem}

Let $P$ be the matrix 
$P=\begin{pmatrix}
0_{2m \times n} & I_{2m}\\
I_{n} & 0_{n \times 2m}
\end{pmatrix}$.
Then we have
\begin{align*}
\hspace{-8mm}P^{t}L(v)P&={\scriptsize
	\begin{pmatrix}
	1 & \ldots & 0 & a_{2m} & 0 & 0 & 0 & \ldots & 0 & 0\\
	\vdots & \vdots & \vdots & \vdots & \vdots & \vdots & \vdots & \vdots & \vdots & \vdots\\
	0 & \ldots & 1 & a_{n+2m-1} & 0 & 0 & 0 & \ldots & 0 & 0\\
	0 & \ldots & 0 & 1 & 0 & 0 & 0 & \ldots & 0 & 0\\
	-\bar{1}^{-1} (\overline{a}_{2m} \varphi_{11} & \ldots & -\bar{1}^{-1} ( \overline{a}_{2m} \varphi_{1n} & a_1  & 1 & \overline{a}_{3} & -\bar{1}^{-1} \overline{a}_{2} & \ldots & \overline{a}_{2m-1} & -\bar{1}^{-1} \overline{a}_{2m-2}\\
	+ \ldots +\overline{a}_{n+2m-1} \varphi_{n1}) & & + \ldots +\overline{a}_{n+2m-1} \varphi_{nn}) & & & & & & &\\
	0 & \ldots & 0 & a_2 & 0 & 1 & 0 & \ldots & 0 & 0\\
	0 & \ldots & 0 & a_{3} & 0 & 0 & 1 & \ldots & 0 & 0\\
	\vdots & \vdots & \vdots & \vdots & \vdots & \vdots & \vdots & \vdots & \vdots & \vdots\\
	0 & \ldots & 0 & a_{2m-2} & 0 & 0 & 0 & \ldots & 0 & 0\\
	0 & \ldots & 0 & a_{2m-1} & 0 &  0 & 0 & \ldots & 0 & 1
	\end{pmatrix}}\\
\\
&=T_{-1}(-a_{2m}v_{1}-a_{2m+1}v_{2}-\ldots-a_{n+2m-1}v_{n}-a_2 e_{2}-a_3e_{-2}- \ldots -a_{2m-2}e_{m}-a_{2m-1}e_{-m},a_1), 
\end{align*} 
and
\begin{align*}
P^{t}L(v)^{*}P&={\scriptsize
	\begin{pmatrix}
	1 & \ldots & 0 & 0 & -\bar{1}^{-1} (\overline{a}_{2m}(\varphi^{-1})_{11}+\overline{a}_{2m+1} (\varphi^{-1})_{12} & 0 & 0 & \ldots & 0 & 0\\
	& & & & +\ldots+\overline{a}_{2m-1+n} (\varphi^{-1})_{1n}) & & & & &\\
	\vdots & \vdots & \vdots & \vdots & \vdots & \vdots & \vdots & \vdots & \vdots & \vdots\\
	0 & \ldots & 1 & 0 & -\bar{1}^{-1} (\overline{a}_{2m}(\varphi^{-1})_{n1}+\overline{a}_{2m+1} (\varphi^{-1})_{n2} & 0 & 0 & \ldots & 0 & 0\\
	& & & & +\ldots+\overline{a}_{2m-1+n} (\varphi^{-1})_{nn} & & & & &\\
	a_{2m} & \ldots & a_{n+2m-1} & 1 & a_1 & a_{2} & a_{3} & \ldots & a_{2m-2} & a_{2m-1}\\
	0 & \ldots & 0 & 0 & 1 & 0 & 0 & \ldots & 0 & 0\\
	0 & \ldots & 0 & 0 & (\bar{1}^{-1})^2 \overline{a}_3 & 1 & 0 & \ldots & 0 & 0\\
	0 & \ldots & 0 & 0 & -\bar{1}^{-1}\overline{a}_2 & 0 & 1 & \ldots & 0 & 0\\
	\vdots & \vdots & \vdots & \vdots & \vdots & \vdots & \vdots & \vdots & \vdots & \vdots\\
	0 & \ldots & 0 & 0 & (\bar{1}^{-1})^2 \overline{a}_{2m-1} & 0 & 0 & \ldots & 0 & 0\\
	0 & \ldots & 0 & 0 & -\bar{1}^{-1}\overline{a}_{2m-2} & 0 & 0 & \ldots & 0 & 1
	\end{pmatrix}}\\
\\
&=T_{1}(\bar{1}^{-1} (\overline{a}_{2m} (\varphi^{-1})_{11}+\overline{a}_{2m+1} (\varphi^{-1})_{12}+\ldots+\overline{a}_{2m-1+n}(\varphi^{-1})_{1n})v_{1}\\
& \hspace{2cm}+\ldots+\bar{1}^{-1} (\overline{a}_{2m} (\varphi^{-1})_{n1}+\overline{a}_{2m+1} (\varphi^{-1})_{n2}+\ldots+\overline{a}_{2m-1+n} (\varphi^{-1})_{nn})v_{n}\\
& \hspace{3cm} -(\bar{1}^{-1})^2 \overline{a}_3 e_{2}+\bar{1}^{-1}\overline{a}_2 e_{-2}- \ldots -(\bar{1}^{-1})^2 \overline{a}_{2m-1}e_{m}+\bar{1}^{-1} \overline{a}_{2m-2}e_{-m}),-\bar{1} a_1).
\end{align*} 
Also by direct computations, we can verify that $(u_1,a_1)$ and $(u_2,-\bar{1} a_1)$ belongs to $\mathfrak L_{\max}$, where
$$u_1 = -a_{2m}v_{1}-a_{2m+1}v_{2}-\ldots-a_{n+2m-1}v_{n}-a_2 e_{2}-a_3e_{-2}- \ldots -a_{2m-2}e_{m}-a_{2m-1}e_{-m} \mbox{  and }$$
\begin{align*}
u_2&=\bar{1}^{-1} (\overline{a}_{2m} (\varphi^{-1})_{11}+\overline{a}_{2m+1} (\varphi^{-1})_{12}+\ldots+\overline{a}_{2m-1+n}(\varphi^{-1})_{1n})v_{1}\\
&\hspace{2cm}+\ldots+\bar{1}^{-1} (\overline{a}_{2m} (\varphi^{-1})_{n1}+\overline{a}_{2m+1} (\varphi^{-1})_{n2}+\ldots+\overline{a}_{2m-1+n} (\varphi^{-1})_{nn})v_{n}\\
&\hspace{3cm} -(\bar{1}^{-1})^2 \overline{a}_3 e_{2}+\bar{1}^{-1}\overline{a}_2 e_{-2}- \ldots -(\bar{1}^{-1})^2 \overline{a}_{2m-1}e_{m}+\bar{1}^{-1} \overline{a}_{2m-2}e_{-m}).
\end{align*}
Thus $P^{t}L(v)P$ and $P^{t} L(v)^{*}P$ belongs to ${\rm EU}_{2m}(R, \mathfrak L_{\max})$. 

\vspace{3mm}
Now to prove the reverse inclusion take an arbitrary element in ${\rm EU}_{2m}(R, \mathfrak L_{\max})$. By Proposition \ref{B1}, the elements of ${\rm EU}_{2m}(R, \mathfrak L_{\max})$ are generated by elements of the form $T_{\pm i}(v,a)$ where $(v,a) \in \mathfrak{L}_{\max}$ and $\langle e_1, v \rangle=\langle e_{-1}, v \rangle=0$. Therefore $v$ has the form
$$v=t_{1}v_{1}+ \ldots +t_{n}v_{n}+b_{2}e_{2}+b_{-2}e_{-2}+\ldots+b_{m}e_{m}+b_{-m}e_{-m}.$$ Then 
$T_{1}(v,a)=P^t R(u)P$,
where 
\begin{align*}
u&=(-\bar{1}^{-1}a,\bar{1}^{-1}\overline{b}_{-2},-(\bar{1}^{-1})^2 \overline{b}_{2},\ldots,\bar{1}^{-1}\overline{b}_{-m},-(\bar{1}^{-1})^2 \overline{b}_{m},(\bar{1}^{-1})^2 ( \overline{t}_1 \varphi_{11} + \ldots +\overline{t}_n \varphi_{1n}),\\
& \hspace{6cm} \ldots,(\bar{1}^{-1})^2 ( \overline{t}_1 \varphi_{n1} + \ldots +\overline{t}_n \varphi_{nn}))
\end{align*}
and $T_{-1}(v,a)=P^t L(u)P$, where $u=(a,-b_{2},-b_{-2},\ldots,-b_{m},-b_{-m},-t_{1},\ldots,-t_{n}).$  

\vspace{3mm}
Therefore the group generated by $L(v)$ and $L(v)^{*}$, for $v \in R^{2m+n-1}$ is congruent to the Petrov's odd elementary hyperbolic unitary group ${\rm EU}_{2m}(R, \mathfrak L_{\max})$.

\begin{remark}
	The matrix $P$ we got in Theorem \ref{A1} is invertible and $P^{-1}=P^t$. Thus the two groups considered here are conjugate subgroups of $\rm{GL}_{n+2m}(R).$ 
\end{remark}

\begin{example}
	Here we illustrate Theorem \ref{A1} in which in the special case where $R$ is a commutative ring and the involution is $a \mapsto \bar{a}=-a$. In this case, the anti-Hermitian form becomes Hermitian. The form matrix for Petrov's group in this case is $\varPsi=\widetilde{\psi}_{m} \perp \varphi$, where $\varphi$ is the form matrix for $V_0$ and $\widetilde{\psi}_{r}$ for $r>1$ is defined as follows.
	
	\vspace{3mm}
	
	Let $\widetilde{\psi_{1}}=\begin{pmatrix}
	0 & 1\\
	1 & 0
	\end{pmatrix}$
	and $\widetilde{\psi}_{r}=\widetilde{\psi}_{1} \perp \widetilde{\psi}_{r-1}$ for $r >1$. The form matrix $\varPsi$ is of the form 
	$\begin{pmatrix}
	0 & c\\
	c^{t} & \mu
	\end{pmatrix}$, where 
	$c=\begin{pmatrix}
	1 & \ldots & 0
	\end{pmatrix} \in R^{n+2m-1}$
	and $\mu$ is the $(n+2m) \times (n+2m)$ matrix given by $\mu=\begin{pmatrix}
	0 & 0 & 0\\
	0 & \widetilde{\psi}_{m-2} & 0\\
	0 & 0 & \varphi
	\end{pmatrix}$, and $\varPsi^{-1}$ is of the form
	$\begin{pmatrix}
	0 & d\\
	d^{t} & \rho
	\end{pmatrix}$, where 
	$d=\begin{pmatrix}
	1 & \ldots & 0
	\end{pmatrix} \in R^{n+2m-1}$ and 
	$\rho=\begin{pmatrix}
	0 & 0 & 0\\
	0 & \widetilde{\psi}_{m-2} & 0\\
	0 & 0 & \varphi^{-1}
	\end{pmatrix}$. 
	
	\vspace{2mm}
	
	Thus for $v=(a_1, \ldots, a_{n+2m-1}) \in R^{n+2m-1}$, we get the matrices $\alpha$ and $\beta$ as follows:
	$$\alpha =I_{n+2m-1}-d^t v \mu \mbox{ and } \beta = I_{n+2m-1} - \rho v^t c.$$      
	
	\vspace{2mm}
	
	The matrices $\alpha$ and $\beta$ are of the following form.
	$$\alpha = {\scriptsize{\begin{pmatrix}
			1 & -a_3 & -a_2 & \ldots & -a_{2m-1} & -a_{2m-2} & -a_{2m}\varphi_{11}-a_{2m+1} \varphi_{21} & \ldots & -a_{2m}\varphi_{1n}-a_{2m+1} \varphi_{2n}\\
			& & & & & & \hspace{5mm}-\ldots-a_{2m-1+n} \varphi_{n1} & & \hspace{5mm}-\ldots-a_{2m-1+n} \varphi_{nn}\\
			0 & 1 & 0 & \ldots & 0 & 0 & 0 & \ldots & 0\\
			0 & 0 & 1 & \ldots & 0 & 0 & 0 & \ldots & 0\\
			\vdots & \vdots & \vdots & \vdots & \vdots & \vdots & \vdots & \vdots & \vdots\\
			0 & 0 & 0 & \ldots & 0 & 0 & 0 & \ldots & 1
			\end{pmatrix}}},$$
	
	$$\mbox{ and }\qquad \beta = {\scriptsize{\begin{pmatrix}
			1 & 0 & 0 & \ldots & 0\\
			-a_3 & 1 & 0 & \ldots & 0\\
			-a_2 & 0 & 1 & \ldots & 0\\
			\vdots & \vdots & \vdots & \vdots & \vdots\\
			-a_{2m-1} & 0 & 0 & \ldots & 0\\
			-a_{2m-2} & 0 & 0 & \ldots & 0\\
			a_{2m}(\varphi^{-1})_{11}+a_{2m+1} (\varphi^{-1})_{12}+\ldots+a_{2m-1+n} (\varphi^{-1})_{1n} & 0 & 0 & \ldots & 0\\
			\vdots & \vdots & \vdots & \vdots & \vdots \\
			a_{2m}(\varphi^{-1})_{n1}+a_{2m+1} (\varphi^{-1})_{n2}+\ldots+a_{2m-1+n} (\varphi^{-1})_{nn} & 0 & 0 & \ldots & 1
			\end{pmatrix}}}.$$
	
	\vspace{3mm}
	
	Therefore, we have the  matrices $L(v)$ and $L(v)^{*}$ of the form
	$$L(v)= {\scriptsize{\begin{pmatrix}
			1 & 0 & 0 & 0 & \ldots & 0 & 0 & 0 & \ldots & 0\\
			a_1 & 1 & -a_3 & -a_2 & \ldots & -a_{2m-1} & -a_{2m-2} & -a_{2m}\varphi_{11}-a_{2m+1} \varphi_{21} & \ldots & -a_{2m}\varphi_{1n}-a_{2m+1} \varphi_{2n}\\
			& & & & & & & -\ldots-a_{2m-1+n} \varphi_{n1} & & -\ldots-a_{2m-1+n} \varphi_{nn}\\ 
			a_2 & 0 & 1 & 0 & \ldots & 0 & 0 & 0 & \ldots & 0\\
			a_3 & 0 & 0 & 1 & \ldots & 0 & 0 & 0 & \ldots & 0\\
			\vdots & \vdots & \vdots & \vdots & \vdots & \vdots & \vdots & \vdots & \vdots & \vdots\\
			a_{n+2m-1} & 0 & 0 & 0 & \ldots & 0 & 0 & 0 & \ldots & 1
			\end{pmatrix}}} $$
	
	$$\mbox{ and }L(v)^{*}={\scriptsize{\begin{pmatrix}
			1 & a_1 & a_2 & \ldots & a_{n+2m-1}\\
			0 & 1 & 0 & \ldots & 0\\
			0 & -a_3 & 1 & \ldots & 0\\
			0 & -a_2 & 0 & \ldots & 0\\
			\vdots & \vdots & \vdots & \vdots & \vdots \\
			0 & -a_{2m-1} & 0 & \ldots & 0\\
			0 & -a_{2m-2} & 0 & \ldots & 0\\
			0 & a_{2m}(\varphi^{-1})_{11}+a_{2m+1} (\varphi^{-1})_{12}+\ldots+a_{2m-1+n} (\varphi^{-1})_{1n} & 0 & \ldots & 0\\
			\vdots & \vdots & \vdots & \vdots & \vdots\\
			0 & a_{2m}(\varphi^{-1})_{n1}+a_{2m+1} (\varphi^{-1})_{n2}+\ldots+a_{2m-1+n} (\varphi^{-1})_{nn} & 0 & \ldots & 1
			\end{pmatrix}}}.$$
	
	\vspace{3mm}
	
	By Remark \ref{F}, both the matrices $L(v)$ and $L(v)^{*}$ in this example  belongs to ${\rm U}_{2m}(R, \mathfrak L_{\max})$. 
	
	\vspace{2mm}
	Now let $P$ be the matrix 
	$P=\begin{pmatrix}
	0_{2m \times n} & I_{2m}\\
	I_{n} & 0_{n \times 2m}
	\end{pmatrix}$.
	We have
	
	\begin{align*}
	P^{t}L(v)P&={\scriptsize
		\begin{pmatrix}
		1 & \ldots & 0 & a_{2m} & 0 & 0 & 0 & \ldots & 0 & 0\\
		\vdots & \vdots & \vdots & \vdots & \vdots & \vdots & \vdots & \vdots & \vdots & \vdots\\
		0 & \ldots & 1 & a_{n+2m-1} & 0 & 0 & 0 & \ldots & 0 & 0\\
		0 & \ldots & 0 & 1 & 0 & 0 & 0 & \ldots & 0 & 0\\
		a_{2m}\varphi_{11}+\ldots & \ldots & a_{2m}\varphi_{1n}+\ldots & a_1 & 1 & -a_{3} & -a_{2} & \ldots & -a_{2m-1} & -a_{2m-2}\\
		+a_{n+2m-1} \varphi_{n1} & & +a_{n+2m-1} \varphi_{nn} & & & & & & &\\
		0 & \ldots & 0 & a_2 & 0 & 1 & 0 & \ldots & 0 & 0\\
		0 & \ldots & 0 & a_3 & 0 & 0 & 1 & \ldots & 0 & 0\\
		\vdots & \vdots & \vdots & \vdots & \vdots & \vdots & \vdots & \vdots & \vdots & \vdots\\
		0 & \ldots & 0 & a_{n+2m-2} & 0 & 0 & 0 & \ldots & 0 & 0\\
		0 & \ldots & 0 & a_{n+2m-1} & 0 & 0 & 0 & \ldots & 0 & 1
		\end{pmatrix}}\\
	\\
	&=T_{-1}(-a_{2m}v_1- \ldots - a_{n+2m-1} v_n -a_2 e_2 -a_3 e_{-2}- \ldots -a_{2m-2}e_{m}-a_{2m-1}e_{-m},a_1), \mbox{ and }
	\end{align*}
	
	\begin{align*}
	P^{t}L(v)^{*}P &=
	{\scriptsize 
		\begin{pmatrix}
		1 & \ldots & 0 & 0 & -(a_{2m}(\varphi^{-1})_{11}+ \ldots a_{n+2m-1}(\varphi^{-1})_{n1}) & 0 & 0 & \ldots & 0 & 0\\
		\vdots & \vdots & \vdots & \vdots & \vdots & \vdots & \vdots & \vdots & \vdots & \vdots\\
		0 & \ldots & 1 & 0 & -(a_{2m}(\varphi^{-1})_{1n}+ \ldots -a_{n+2m-1}(\varphi^{-1})_{nn}) & 0 & 0 & \ldots & 0 & 0\\
		a_{2m} & \ldots &a_{n+2m-1} & 1 & a_1 & a_{2} & a_{3} & \ldots & a_{2m-2} & a_{2m-1}\\
		0 & \ldots & 0 & 0 & 1 & 0 & 0 & \ldots & 0 & 0\\
		0 & \ldots & 0 & 0 & -a_{3} & 1 & 0 & \ldots & 0 & 0\\
		0 & \ldots & 0 & 0 & -a_{2} & 0 & 1 & \ldots & 0 & 0\\
		\vdots & \vdots & \vdots & \vdots & \vdots & \vdots & \vdots & \vdots & \vdots & \vdots\\
		0 & \ldots & 0 & 0 & -a_{2m-1} & 0 & 0 & \ldots & 0 & 0\\
		0 & \ldots & 0 & 0 & -a_{2m-2} & 0 & 0 & \ldots & 0 & 1
		\end{pmatrix}} \\
	\\
	&=T_1((a_{2m}(\varphi^{-1})_{11}+ \ldots a_{n+2m-1}(\varphi^{-1})_{n1})v_{1}+\ldots+ (a_{2m}(\varphi^{-1})_{1n}\\
	&\hspace{8mm}+ \ldots +a_{n+2m-1}(\varphi^{-1})_{nn})v_{n}+a_3 e_{2}+a_2e_{-2}+ \ldots +a_{2m-1}e_{m}+a_{2m-2}e_{-m},a_1).
	\end{align*}
	
	\noindent Also, $(u_1,a_1)$ and $(u_2,a_1)$ belongs to $\mathfrak L_{\max}$, where 
	$$u_1=-a_{2m}v_1- \ldots - a_{n+2m-1} v_n -a_2 e_2 -a_3 e_{-2}- \ldots -a_{2m-2}e_{m}-a_{2m-1}e_{-m}~ \mbox{ and }$$ 
	\begin{align*}
	u_2&=a_{2m}(\varphi^{-1})_{11}+ \ldots +a_{n+2m-1}(\varphi^{-1})_{n1})v_{1}+\ldots+ (a_{2m}(\varphi^{-1})_{1n}+ \ldots +a_{n+2m-1}(\varphi^{-1})_{nn})v_{n}\\
	& \hspace{2cm} + a_3 e_{2}+a_2e_{-2}+ \ldots +a_{2m-1}e_{m}+a_{2m-2}e_{-m}.
	\end{align*}
	Thus $P^{t}L(v)P$ and $P^{t} L(v)^{*}P$ belongs to ${\rm EU}_{2m}(R, \mathfrak L_{\max})$. 
	
	\vspace{2mm}
	
	Now for the reverse inclusion take an arbitrary element in ${\rm EU}_{2m}(R, \mathfrak L_{\max})$. By Proposition \ref{B1}, the elements of ${\rm EU}_{2m}(R, \mathfrak L_{\max})$ are generated by elements of the form $T_{\pm i}(v,a)$ where $(v,a) \in \mathfrak L_{\max}$ and $\langle e_1, v \rangle=\langle e_{-1}, v \rangle=0$. Therefore $v$ has the form
	$v=t_{1}v_{1}+ \ldots +t_{n}v_{n}+b_{2}e_{2}+b_{-2}e_{-2}+\ldots+b_{m}e_{m}+b_{-m}e_{-m}$. Then we have
	\begin{align*}
	T_1(v,a)&=P^t L((a,b_{-2},b_{2},\ldots,b_{-m},b_{m},t_{1},\ldots,t_{n}))P, ~\mbox{  and }\\
	\vspace{2mm}
	T_{-1}(v,a)&=P^t R((a,b_{2},b_{-2},\ldots,b_{m},b_{-m},\varphi_{11}t_1+\ldots+\varphi_{n1}t_{n}, \ldots, \varphi_{1n}t_1+\ldots+\varphi_{nn}t_{n}))P.
	\end{align*}
	
	Therefore the group generated by $L(v)$ and $L(v)^{*}$ for $v \in R^{2m+n-1}$ is congruent to the Petrov's group ${\rm EU}_{2m}(R, \mathfrak L_{\max})$.
\end{example}

\section*{Acknowledgements}
 The first author would like to acknowledge the support by ``Seed Money for New Research Initiatives [2017-18]'', Cochin University of Science and Technology. The second author thank Department of Science and Technology (DST) for the INSPIRE fellowship which supported this work.

\bibliographystyle{amsplain}

\end{document}